\documentclass[11pt]{article}%
\usepackage{epsfig}
\usepackage{amsmath}
\usepackage{graphicx}
\usepackage{graphicx,color}
\usepackage{amsfonts}
\usepackage{amssymb}
\usepackage{amsmath, amssymb ,amsthm, amsfonts, amsgen}
\usepackage{graphicx}%
\providecommand{\U}[1]{\protect\rule{.1in}{.1in}}
\newcommand{\e}{\varepsilon}

\newcommand{\R}{\mathbb{R}}

\newtheorem{theorem}{Theorem}
\newtheorem{proposition}{Proposition}[section]

\newtheorem{remark}{Remark}[section]
\newtheorem{lemma}{Lemma}[section]
\numberwithin{equation}{section}

\usepackage[utf8]{inputenc}

\begin{document}

\title{A Liouville type result and quantization effects on the system $-\Delta u =  u J'(1-|u|^{2})$ for a potential convex near zero}
\date{} 
\author{U. De Maio\thanks{Universit\`a Federico II di Napoli, Dipartimento di
Matematica R. Caccioppoli, Via Cinthia, Monte S. Angelo, Napoli, Italia. E-mail: udemaio.unina.it} $\,$, and
R. Hadiji \thanks{Univ Paris Est Creteil, CNRS, LAMA, F-94010 Creteil, France. \newline \indent \,
 Univ Gustave Eiffel, LAMA, F-77447 Marne-la-Vall\'ee, France. E-mail: rejeb.hadiji@u-pec.fr}  
  $\,$,and
C. Lefter\thanks{Al.I.Cuza University, Faculty of Mathematics and Octav Mayer Institute of Mathematics, Ia\c{s}i, Romania E-mail: catalin.lefter@uaic.ro}$\,$, and  
C. Perugia\thanks{Universit\`a del Sannio, Dipartimento di Scienze e Tecnologie, Via de Sanctis, 82100, Benevento, Italia. E-mail: cperugia@unisannio.it}$\,$ }
\maketitle

\begin{abstract}
	{We consider a Ginzburg-Landau type equation in $\R^2$ of the form 
	$-\Delta u =  u J'(1-|u|^{2})$ with a potential function $J$ satisfying weak conditions allowing for example a zero of infinite order in the origin. We extend in this context the results concerning quantization of finite potential solutions of H.Brezis, F.Merle, T.Rivière from \cite{BMR} who treat the case when $J$ behaves polinomially near 0, as well as a result of Th. Cazenave, found in the same reference, and concerning the form of finite energy solutions.
}


  \medskip


\end{abstract}

\section{Introduction}
We consider in the paper  the following system: 
\begin{equation}
\label{p1}
-\Delta u=u j(1-|u|^{2})\,\,\, \hbox{in} \, \R^{2},
\end{equation}
with $u: \R
^{2}\rightarrow\R^{2}$, where $j\left( t\right) =\dfrac{dJ}{dt}(t)  $   with a $C^{2}$ functional $J:\R\rightarrow[0,\infty)$  satisfying the following conditions:
\begin{itemize}
	\item[($\mathit{H1}$)]	$J(0)=0$ and $J(t)>0$, $\forall t\in(0,\infty)$.
	\item [($\mathit{H2}$)]$\displaystyle \liminf_{t\rightarrow -\infty}{J(t)\over \sqrt{\vert t\vert}}>K$, for some $K>0$.
	\item [($\mathit{H3}$)] $j(t) < 0$ on $(-\infty,0]$ and $j(t) > 0$ on $(0,1]$.
	\item[($\mathit{H4}$)] There {exists  }  $\eta_0>0 $ such that $$J''(t)={\dfrac{d^2J}{dt^2}(t)} \ge 0, \text{ for } t\in[0,\eta_0).$$
\end{itemize}

Solutions to \eqref{p1} are formal critical points of the functional:
\begin{equation}
\label{energyfunct}
E(u) = \int_{\R^2} (|\nabla u(x)|^2+J(1-|u(x)|^2)dx.
\end{equation}

In this paper we generalize two results  from \cite{BMR} for which one needs to use and to properly adapt some of the ideas from their proofs. 

The first result concerns a Liouville-type theorem:

\begin{theorem}\label{teo1} 
	Under assumptions ($\mathit{H1}$)$\div$($\mathit{{H4}}$), let $u:\R
	^{\textcolor{black}{2}}\rightarrow\R^{\textcolor{black}{2}}$ be a smooth solution of
	\eqref{p1} satisfying
	\begin{equation}
	\int_{\R^2}|\nabla u|^{2}dx<+\infty.\label{c1}
	\end{equation}
	Then either $u\equiv0$ or $u$ is a constant $C$ with $|C|=1$ on $\R
	^{{2}}$.
\end{theorem}
This result appeared in Brezis, Merle and Rivière in \cite{BMR} for the classic Ginzburg-Landau system, namely $J(t)=t^{2}$. 

Concerning the   problem formulated in $\R^n$ we mention \textcolor{black}{some} references \textcolor{black}{in the following lines}. In \cite{A}, Alikakos showed  that if $u:\R^{n}\rightarrow\R^{n}$ is a solution of $(\ref{p1})$ such that $\displaystyle\int_{\R^{n}}\left| \nabla u\right|^{2}+\int_{\R^{n}} J\left( 1-|u|^{2}\right)  \leq C$,  then $u$ is a constant. In \cite{F1}, Farina showed a Liouville-type theorem  for the Ginzburg-Landau system which asserts that every solution $u\in C^{2}\left( \R^{n},\R^{m}\right) $ satisfying $\displaystyle\int_{\R^{n}} \left( 1-|u|^{2}\right)^{2}<+\infty$ is constant (and of unit norm), provided $n\geq4$ ($m\geq1$). Moreover, in \cite{F2}, Farina showed a Liouville-type theorem  for a variant of  the Ginzburg-Landau system under assumption that the solution $u$ satisfies $\displaystyle\int_{\R^{n}}\left| \nabla u\right|^{2}<+\infty$ for $n\geq 2$.

Our second  result generalizes the  theorem of quantization of the energy obtained by Brezis, Merle and Rivière in
\cite{BMR}:
\begin{theorem} \label{teo2} 
	Under assumptions ($\mathit{H1}$)$\div$($\mathit{{H4}}$), let $u:\R%
	^{2}\rightarrow\R^{2}$  be a smooth solution of equation $\left(
	\ref{p1}\right)  $. Then
	\begin{equation}
	\int_{\R^{2}}J\left(  1-\left\vert u\right\vert ^{2}\right)  =\pi
	d^{2}\label{pot}%
	\end{equation}
	for some integer $d=0,1,2,..,\infty$.
\end{theorem}

The hypotheses on $J$ include the classic situation $J(t)=t^{2}$ considered in  \cite{BMR}   
and include also a large class of functionals $J$  having a zero of infinite order at $t=0$, such as
\begin{equation}
\label{J.1}J(t)=
\begin{cases}
\exp\left( -1/t^{k}\right)  & \mbox{if } t >  0,\\
0  & \mbox{if } t =  0,\\
{ \sqrt{\vert t\vert} }g(t) & \mbox{if } t <  0,
\end{cases}
\end{equation}
for some $k\geq 1$ and a $C^2$ positive non increasing function $g:(-\infty,0]\rightarrow\R$, $g'(0)=g''(0)=0$. 
A key point in our approach refers to the proof of the bound $|u|\leq1$ which does not follow the lines in \cite{BMR}.

Several results regarding the study of Ginzburg-Landau energy are present in literature. Béthuel-Brezis and Hélein were the first \textcolor{black}{to} study the asymptotic behabiour of the standard energy of Ginzburg-Landau in bounded domain in   \cite{BBH1} and \cite{BBH2}. Generalized vortices in the magnetic Ginzburg-Landau model is considered in  \cite{SaSe}.
In \cite{ES}, Sandier showed that every locally minimizing solution of $-\Delta u=u(1-|u|^{2})$ has a bounded potential. 

The presence of a nonconstant weight function which is motivated by the problem of pinning of vortices, that is, of forcing the location of the vortices to some restricted sites\textcolor{black}{,} are studied in \textcolor{black}{\cite{AS2}}, \cite{BH1}, \cite{BH2}, \cite{BH3}, \cite{HP}, \cite{DuG} and \cite{R}. 

In  \cite{HS1}, the authors investigated the energy of a Ginzburg-Landau type energy with potential having a zero of infinite order. The showed that the main difference with respect to the usual GL-energy is in the  asymptotic of the energy. For $J$ with a zero of infinite order the ``energy cost'' of a degree-one vortex may be much less than
 the cost of $2\pi\log\frac{1}{\e}$ for the GL-functional .
 
In \cite{HP2020}, the authors  considered the asymptotic behaviour of minimizing solutions of a  Ginzburg-Landau type functional with a positive weight and with convex potential near $0$ and they estimated the energy in this case. They also generalized the lower bound for the energy for the Ginzburg-Landau energy of unit vector field given initially by Brezis-Merle-Riviière in  \cite{BMR} for the case where \textcolor{black}{the} potential \textcolor{black}{has} a zero of infinite order.

In \cite{F2}, Farina showed a quantization effect for a variant of the Ginzburg-Landau system  under the assumption that $\displaystyle\int_{\R^{2}} P_h\left(|u|^{2}\right)<+\infty$, where the potential $P_h$ is given by $P_h(t)=\dfrac{1}{2}\prod_{j=1}^{h}\left( 1-t^{2}-k_{j}\right) ^{2}$.\\

Note that  
starting from the Ginzburg-Landau type problem
\begin{equation}
\label{GL}
\begin{array}
{ll}
-\Delta u_{\varepsilon}=\dfrac{1}{\varepsilon^{2}}u_{\varepsilon}
j(1-|u_{\varepsilon}|^{2}) & $in$\,\, \Omega_{\varepsilon},
\end{array}
\end{equation}
after a blow-up argument we are led to $(\ref{p1})$.\\

Problem $(\ref{GL})$, defined in a bounded domain and with potential
satisfying conditions ($\mathit{H1}$), first part of   ($\mathit{H3}
$) and ($\mathit{H4}$) was considered in \cite{HS1}, where the asymptotic behaviour of the minimizers and their energies is described.\\


\section{Preliminary results}

In this section we want to prove some results which will be useful in the sequel. \textcolor{black}{Moreover here and after, for every $R>0$ and $x_{0}\in \R^{n}$, we will denote with
$B_{R}\left(x_{0}\right)$ the set $\left\{ x\in\R^{n};|x-x_{0}|<R \right\}$, with $B_{R}$ the set $\left\{ x\in\R^{n};|x|<R \right\}$ and with $S_{R}$ its
boundary. Moreover, for every function $v$,  we will denote $v^{+}(x)$=$\max\{v(x),0\}$ and  $v^-=-\min\{v,0\}$}.

\begin{lemma}\label{lemma1}
		Let $\Omega\subset\R^n$ be open, not necessarily bounded, connected domain with smooth boundary $\partial\Omega\not=\emptyset$. 
		
		Suppose $c(\textcolor{black}{x})\in L^\infty_{loc}(\Omega)$, $c(x)\ge 0$ in $\Omega$ and $\delta \textcolor{black}{\geq} 0$ are given. Let $\rho\in H^2_{loc}(\Omega)$ be a strong solution to the problem:
		\begin{equation}\label{eqrho}
		\left\{
		\begin{array}{ll}
		-\Delta\rho+c(x)\rho=0, &x\in\Omega,\\
		\rho(x)=\delta\ge0, &x\in\partial\Omega		\end{array}
		\right.
		\end{equation}
and $\rho$ satisfies the lower bound	
$$
\rho(x)\ge\delta,\,\forall x\in\Omega.
$$
Then the following statements hold:
\begin{enumerate}

\item \textcolor{black}{Suppose that $\displaystyle \int_\Omega|\nabla \rho|^2dx<\infty$ and that $n=2$ if $\Omega$ is unbounded}. Then,  $\rho(x)\equiv\delta$ in $\Omega$. If $c\not\equiv0$ then $\delta=0$.

\item Suppose that  $\rho\in L^1(\Omega)$. If $\delta>0$ or if $\delta=0$ and $\mu(\Omega)<\infty$, then $\rho(x)\equiv\delta$ and if $\delta>0$ then $c(x)\equiv 0$.
\end{enumerate}

	\end{lemma}
	
\proof

(1) Suppose first that $\displaystyle \int_\Omega|\nabla \rho|^2dx<\infty$.

If $\Omega$ is bounded multiply  the equation by $\rho$, integrate on $\Omega$ and obtain:
$$
\int_\Omega|\nabla\rho|^2dx-\int_{\partial\Omega}\frac{\partial \rho}{\partial\nu}\rho d\sigma+\int_\Omega c(x)\rho^2dx=0.
$$
As $\rho|_{\partial\Omega}=\delta$ and $\rho\ge\delta$ in $\Omega$ we have that $\displaystyle\frac{\partial\rho }{\partial \nu}\leq 0$ on $\partial\Omega$. Consequently, from the above equality we deduce that $c(x)\rho\equiv0$ in $\Omega$ and $\nabla \rho\equiv 0$. The latter implies $\rho\equiv\delta$ in $\Omega$. Now, as $c(x)\rho\equiv0$, if $c(x)\not\equiv0$ then necessarily $\delta=0$.

 Suppose now that $\Omega$ is unbounded.
 
 Consider $\eta:\R^2\rightarrow \R$ a $C^\infty$ function satisfying:
 $$
 \eta\ge0\text{ in }\R^{\textcolor{black}{2}},\,\eta(x)=1\text{ for }|x|<\frac12 \text{ and }\eta(x)=0\text{ for }|x|>1,
 $$
 and for $R>0$ define $$\eta_R(x)=\eta\left(\frac{x}{R}\right).$$ Observe that for $K=\sup|\nabla\eta|<\infty$ we have
 \begin{equation}\label{g1}
 	\textcolor{black}{\left\{\begin{array}{ll}|\nabla\eta_R(x)|\le\dfrac{K}{R},\,\,\forall x\in\R^2,\vspace{0,3cm}\\
 |\nabla\eta_R(x)|=0, \hbox{ for } |x|\not\in\bigg(\dfrac{R}{2},R\bigg).\end{array}\right.}\end{equation}

Denote by $\Omega_R=\Omega\cap B_R$ and by $A_{\frac{R}{2},R}=B_{R}\setminus B_{\frac{R}{2}}$.

Multiply  equation \eqref{eqrho} by $\eta_R$ and \textcolor{black}{integrating} on $\Omega_R$ to obtain by Green formula:
 \begin{equation}\label{est2}
 \int_{\Omega_R}c(x)\rho(x)\eta_R(x)dx-\int_{\partial\Omega\cap B_R}\frac{\partial\rho}{\partial\nu}\eta_Rd\sigma=-\int_{\Omega_R}\nabla\rho\cdot\nabla\eta_Rdx=-\int_{\Omega\cap A_{\frac{R}{2},R}}\nabla\rho\cdot\nabla\eta_Rdx.
 \end{equation}
 Now, by Cauchy-Schwarz inequality, we have
 $$
 \left|\int_{\Omega\cap A_{\frac{R}{2},R}}\nabla\rho\cdot\nabla\eta_Rdx\right|\le \left(\int_{\Omega\cap A_{\frac{R}{2},R}}|\nabla\rho|^2dx\right)^{\frac12}\frac{K}{R}(\mu(\Omega\cap A_{\frac{R}{2},R}))^{\frac12}.
 $$
 Observe that, as $\displaystyle\int_\Omega|\nabla\rho|^2dx<\infty$ we have
 $$
 \lim_{R\rightarrow\infty} \int_{\Omega\cap A_{\frac{R}{2},R}}|\nabla\rho|^2dx=0
 $$
and, as $\mu(\Omega\cap A_{\frac{R}{2},R})\le \mu( A_{\frac{R}{2},R})=\dfrac{3\pi R^2}{4}$ , passing to the limit for $R\rightarrow\infty$ in \eqref{est2}, we obtain:
$$
\int_{\Omega }c(x)\rho(x) dx-\int_{\partial\Omega  }\frac{\partial\rho}{\partial\nu} d\sigma=0.
$$
As $\rho$ attains its infimum $\delta$ in any point on $\partial\Omega$ we have $\frac{\partial\rho}{\partial\nu}\le0$ on $\partial\Omega$ and so
$$
c(x)\rho(x)\equiv0 \text{ in  }\textcolor{black}{\Omega}.
$$
\textcolor{black}{Consequently $\rho$ is an harmonic function in $\Omega$, hence it attains its minimum $\rho=\delta$ on $\partial\Omega$ and by Hopf maximum principle
$$
\rho\equiv \delta \text{ in }\Omega.
$$
}

(2). Assume now that $\rho\in L^1(\Omega)$. 

Observe that if $\delta>0$ then $\Omega$ has finite Lebesgue measure $\mu(\Omega)<+\infty$ (for $\delta =0$ this is an hypothesis). Indeed, this is a consequence of the fact that $\rho\in L^1(\Omega)$ and $\rho\ge \delta$ in $\Omega$ and thus the constant $\delta$ is in $L^1$, which is $\mu(\Omega)<\infty.$
\medskip

We distinguish two cases: $\Omega$ bounded and $\Omega$ unbounded.\medskip

If $\Omega$ is bounded multiply  the equation by $\rho$, integrate on $\Omega$ and obtain:
$$
\int_\Omega|\nabla\rho|^2dx-\int_{\partial\Omega}\frac{\partial \rho}{\partial\nu}\rho d\sigma+\int_\Omega c(x)\rho^2dx=0.
$$
As $\rho|_{\partial\Omega}=\delta$ and $\rho\ge\delta$ in $\Omega$ we have that $\displaystyle\frac{\partial\rho }{\partial \nu}\leq 0$ on $\partial\Omega$. Consequently, from the above equality we deduce that $c(x)\rho\equiv0$ in $\Omega$, which means $c(\cdot)\equiv0$, and $\nabla \rho\equiv 0$, which implies $\rho\equiv\delta$ in $\Omega$.
\medskip
\medskip

Suppose now that $\Omega$	is an unbounded domain.

For $R>0$ denote by $\Omega_R=\Omega\cap B_R$.
Denote by $\varphi:\overline B_1\rightarrow \R$ the function
$$
\varphi(x)=\frac{1-|x|^2}{2n},
$$
which is in fact the solution to the boundary value problem
$$
\left\{
\begin{array}{ll}
-\Delta\varphi=1&\text{ in } B_1,\\
\varphi=0&\text{ on } \partial B_1.
\end{array}
\right.
$$
Observe that 
$$
|\nabla\varphi(x)|\le C_1:=\frac{1}{n}\text{ for }x\in\textcolor{black}{B_1}
$$
and on the boundary the normal derivative $\displaystyle\frac{\partial\varphi}{\partial\nu}\equiv -C_1$.
\textcolor{black}{Denote by $\varphi_R: B_R\rightarrow\R$ the function defined as}
$$
\varphi_R(x)=\varphi\left(\frac{x}{R}\right),
$$
which is the solution to the boundary value problem 
$$
\left\{
\begin{array}{ll}
\displaystyle -\Delta\varphi=\frac{1}{R^2}&\text{ in } B_R,\\
\varphi=0&\text{ on } S_R,
\end{array}
\right.
$$
and satisfies 
$$
|\nabla\varphi_R(x)|\le \frac{C_1}{R}\text{ in } B_R, \, \frac{\partial\varphi_R}{\partial\nu}\equiv-\frac{C_1}{R}\text{ on }\textcolor{black}{S_R}\textcolor{black}{.}
$$
Denote by $$
\Gamma_1^R=\Omega\cap S_R,\,\Gamma_2^R=\partial\Omega\cap  B_R,
$$
such that 
$$
\partial\Omega_R=\Gamma_1^R\cup\Gamma_2^R.
$$
Observe that, as  $\Omega$ is unbounded and connected,  there exists $R_0>0$ such that for  $R\ge R_0$ both parts of $\partial\Omega_R$ are nonempty\textcolor{black}{,} \textit{i.e.} $\Gamma_1^R\ne\emptyset$ and $\Gamma_2^R\ne\emptyset$.

Multiply now the equation satisfied by $\rho$ with $\varphi_R$, integrate on $B_R$ and use Green formula to obtain:
$$
0=\int_{\Omega_R}(-\Delta\rho(x))\varphi_R(x)dx+\int_{\Omega_R}c(x)\rho(x)\varphi_R(x)dx=
$$
$$
=\int_{\Omega_R}(-\Delta\varphi_R(x))\rho(x)dx+\int_{\partial\Omega_R}\left(\rho\frac{\partial \varphi_R}{\partial\nu}-\varphi_R\frac{\partial \rho}{\partial\nu}\right)d\sigma+ \int_{\Omega_R}c(x)\rho(x)\varphi_R(x)dx
=$$
$$=
\frac{1}{R^2}\int_{\Omega_R}\rho(x)dx-\frac{C_1}{R}\int_{\Gamma_1^R}\rho d\sigma+\delta\int_{\Gamma_2^R}\frac{\partial \varphi_R}{\partial\nu}d\sigma-\int_{\Gamma_2^R}\varphi_R\frac{\partial \rho}{\partial\nu}d\sigma+\int_{\Omega_R}c(x)\rho(x)\varphi_R(x)dx.
$$
Observe that $\displaystyle \frac{\partial \rho}{\partial\nu}\le 0$ on $\Gamma_2^R$ and $c(x)\rho(x)\varphi_R(x)\ge0$ in $\Omega$ and thus:
$$
0\le \frac{1}{R^2}\int_{\Omega_R}\rho(x)dx+ \int_{\Gamma_2^R}\varphi_R\left[-\frac{\partial \rho}{\partial\nu}\right]d\sigma+\int_{\Omega_R}c(x)\rho(x)\varphi_R(x)dx=
$$
\begin{equation}\label{star}
= \delta\int_{\Gamma_2^R}\left[-\frac{\partial \varphi_R}{\partial\nu}\right]d\sigma+ \frac{C_1}{R}\int_{\Gamma_1^R}\rho d\sigma=I_1+I_2.
\end{equation}
Observe now that 
$$
I_1=\delta\int_{\Gamma_2^R}\left[-\frac{\partial \varphi_R}{\partial\nu}\right]d\sigma=
\delta\int_{\partial\Omega_R}\left[-\frac{\partial \varphi_R}{\partial\nu}\right]d\sigma
-\delta\int_{\Gamma_1^R}\left[-\frac{\partial \varphi_R}{\partial\nu}\right]d\sigma=
$$
$$
=\delta\int_{\Omega_R}[-\Delta\varphi_R]dx-\delta \frac{C_1}{R}\int_{\Gamma_1^R}d\sigma=
$$
$$
=\frac{\delta}{R^2}\mu(\Omega_R)-\delta\frac{C_1}{R}\int_{\Gamma_1^R}d\sigma.
$$
Inserting this into \eqref{star} we obtain
$$
0\le \frac{1}{R^2}\int_{\Omega_R}\rho(x)dx+ \int_{\Gamma_2^R}\varphi_R\left[-\frac{\partial \rho}{\partial\nu}\right]d\sigma+\int_{\Omega_R}c(x)\rho(x)\varphi_R(x)dx=I_1+I_2=
$$
\begin{equation}
\label{2star}
=\frac{C_1}{R}\int_{\Gamma_1^R}(\rho-\delta)d\sigma+\frac{\textcolor{black}{\delta}}{R^2}\mu(\Omega_R).
\end{equation}
First, observe that, as $\mu(\Omega)<\infty$,
\begin{equation}\label{11}
\lim_{R\rightarrow+\infty}\frac{1}{R^2}\mu(\Omega_R)=0.
\end{equation}
Then, as $\rho\in L^1(\Omega)$ and $\mu(\Omega)<+\infty$
$$
\int_\Omega(\rho(x)-\delta)dx=\int_{0}^{\infty}\int_{\Gamma_1^R}(\rho-\delta)d\sigma dR<\infty.
$$
Thus, there exists a sequence $R_m\rightarrow+\infty$ such that
\begin{equation}
\label{12}\lim_{R_m\rightarrow+\infty}\int_{\Gamma_1^{R_m}}(\rho-\delta)d\sigma=0.
\end{equation}
Using \eqref{11}, \eqref{12} we obtain, by passing to the limit for $R_m\rightarrow+\infty $ in  \eqref{2star}, that
$$
0=\liminf_{R_m\rightarrow \infty}\left[\frac{1}{R_m^2}\int_{\Omega_{R_m}}\rho(x)dx+ \int_{\Gamma_2^{R_m}}\varphi_{R_m}\left[-\frac{\partial \rho}{\partial\nu}\right]d\sigma+\int_{\Omega_{R_m}}c(x)\rho(x)\varphi_{R_m}(x)dx\right].
$$
By Fatou Lemma we obtain from here that
\begin{equation}
0=\int_{\partial\Omega}\varphi(0)\left[-\frac{\partial \rho}{\partial\nu}\right]d\sigma+
\int_{\Omega}c(x)\rho(x)\varphi (0)dx.
\end{equation}
A first consequence is that $$
c(x)\rho(x)\equiv0\text{ in }\Omega,
$$
and thus $\rho$ is harmonic in $\Omega$ and
 $$
\frac{\partial \rho}{\partial\nu}\equiv0\text{ on }\partial\Omega.
$$
As on $\partial\Omega$ the function $\rho$ attains its minimum $\rho=\delta$, by Hopf maximum principle
$$
\rho\equiv \delta \text{ in }\Omega.
$$

\medskip

\bigskip

\begin{proposition}
	\label{mod1}
	Let $u$ be a classical solution to the problem 
	$$
	-\Delta u =uj(1-|u|^2)\text{ in }\R^2.
	$$ 
	Under the assumption \textit{(H1)}, \textcolor{black}{\textit{(H2)}}, if 
	$$
	\int_{\R^2}|\nabla u|^2dx<\infty
	$$
	or if 
	$$
	\int_{\R^2}J(1-|u|^2) dx<\infty,
	$$
	then 
	$$
	|u(x)|\le 1, \forall x\in \R^2.
	$$
 \end{proposition}
\proof

Denote by 
$$
w=1-|u|^2,
$$
and by $w^+=\max\{w,0\}$, $w^-=-\min\{w,0\}$ such that $w=w^+-w^-$. 

Suppose, by \textcolor{black}{contradiction} that $\sup|u|>1$.

For $\textcolor{black}{\mu}>0$ denote by $A_{\textcolor{black}{\mu}}$ the set
  $$
A_{\textcolor{black}{\mu}}=\{x\in\R^2: w^-(x)=\textcolor{black}{\mu} \}\textcolor{black}{.}
$$
By Sard's lemma we know that almost all $\textcolor{black}{\mu}>0$ in the image of $w^-$ are regular values. So, we may choose such a $\textcolor{black}{\mu}$ and consequently $A_{\textcolor{black}{\mu}}$ is a smooth (not necessarily connected) manifold in $\R^2$ and 
$$
A_{\textcolor{black}{\mu}}=\partial\Omega_{\textcolor{black}{\mu}},\quad \Omega_{\textcolor{black}{\mu}}=\{x\in\R^2: w^-(x)>\textcolor{black}{\mu} \}.
$$

In $\Omega_{\textcolor{black}{\mu}}$ we have $w^-=|u|^2-1$ and $|u|>\sqrt{1+\textcolor{black}{\mu}}$.

In $\Omega_{\textcolor{black}{\mu}}$ we may write locally $u=\rho(x)e^{i\psi(x)} $; the phase  $\psi$ is defined locally up to an additive integer multiple of $2\pi$ but $\nabla \psi$ is defined globally. \textcolor{black}{By separating the real and the imaginary parts, we deduce that }the equation satisfied by $\rho$ is
$$
-\Delta \rho+(|\nabla\psi|^2-j(-w^-))\rho=0\text{ in } \Omega_{\textcolor{black}{\mu}},
$$
with $c(x):=(|\nabla\psi|^2-j(-w^-))\ge0$ in $\Omega_{\textcolor{black}{\mu}}$.

Suppose first that 
$$
\int_{\R^2}|\nabla u|^2dx<\infty.
$$
Observe that in the set $\Omega_{\textcolor{black}{\mu}}$ we have
$$
\nabla u=\nabla\rho e^{i\psi}+i\rho\nabla\psi e^{i\psi},
$$
so $|\nabla u|^2=|\nabla \rho|^2+\rho^2|\nabla\psi|^2\ge|\nabla\rho|^2$. Consequently,
$$
\int_{\Omega_{\textcolor{black}{\mu}}}|\nabla \rho|^2 dx<\infty.
$$
Lemma \ref{lemma1} tells that necessarily $\rho\equiv\sqrt{1+\textcolor{black}{\mu}}$ in $\Omega_{\textcolor{black}{\mu}}$ but this contradicts the definition of $\Omega_{\textcolor{black}{\mu}}=\{x\in\R^2:|u(x)|>\sqrt{1+\textcolor{black}{\mu}}\}$.
\medskip

Suppose now that 
$$
\int_{\R^2}J(1-|u|^2) dx<\infty.
$$
This implies that
$$
\int_{\Omega_\mu}J(-w^-(x)) dx<\infty.
$$
By \textcolor{black}{\textit{(H2)}} we have
$$
\liminf_{t\rightarrow-\infty}\frac{J(t)}{\sqrt{-t}}>0.
$$
So, in $\Omega_{\textcolor{black}{\mu}}$ we have $J(-w^-)=J(1 - |u|^2)\ge \varepsilon |u|$ for some $\varepsilon>0$ small enough.
We deduce that $\rho\in L^1(\Omega_{\textcolor{black}{\mu}})$.

Clearly, $\rho>\sqrt{1+\textcolor{black}{\mu}}$ in $\Omega_{\textcolor{black}{\mu}}$ and $\rho=\sqrt{1+\textcolor{black}{\mu}}$ on $\partial \Omega_{\textcolor{black}{\mu}}$.
The Lemma \ref{lemma1} applied to $\rho$ says that $\rho\equiv \sqrt{1+\textcolor{black}{\mu}}$ in $\Omega_{\textcolor{black}{\mu}}$, but this is a contradiction with the definition of $\Omega_{\textcolor{black}{\mu}}$.

Consequently, we \textcolor{black}{get} $|u|\le 1$ in $\R^2$.

\begin{remark}
For the Ginzburg-Landau system the estimate $\vert u \vert\leq 1$  should  holds true without any additional assumption on the energy terms. Indeed this was proved by Herv\'e and Herv\'e for $ u :  \R^2\rightarrow \R^2$, see \cite{HH}  and later generalized by Brezis for 
$ u :  \R^n\rightarrow \R^k$ using an argument, based on the Keller-Osserman theory, see \cite{B}.  
\end{remark}

\section{Proof of Theorem \ref{teo1}}

The proof of Theorem \ref{teo1} will be performed into \textcolor{black}{two} steps.
\newline

\textbf{Step {1.}} By $(\ref{p1})$ we get
\begin{equation*}
\Delta |u|^2=2|\nabla u|^2+2u\Delta u=2|\nabla u|^2 -2|u|^2j(1-|u|^{2}) \quad \text{in }
\R^{n}
\end{equation*}
hence
\begin{equation}\label{E1}
	|u|^{2}j\left( 1-|u|^{2}\right) =|\nabla u|^{2}-\dfrac{1}{2}%
\Delta|u|^{2}.
\end{equation}
Let $\eta\in C^{\infty}(\R^{n},[0,1])$ satisfy $\eta(x)=1$ for
$|x|\leq1$, and $\eta(x)=0$ for $|x|\geq2$. We set $\eta_{h}=\eta\left(
\dfrac{x}{h}\right) $.
 Multiplying \eqref{E1}  by $\eta_h$ and integrating over $\R^2$ (which is effectively an integral over $B_{2h}$), we
have
\begin{equation}
	\begin{aligned}
		&\int_{\R^2}|u|^2j(1-|u|^2)\eta_h dx\\
		&=\int_{\R^2}|\nabla u|^2\eta_h dx
		+\int_{\R^2}u\nabla
		u\nabla\eta_h dx=\int_{\R^2}|\nabla u|^2\eta_h dx
				+\int_{h\le|x|<2h}u\nabla
				u\nabla\eta_h dx.
	\end{aligned} \label{e4.2}
\end{equation}
For the last term, using that  $|u|\le 1$ and Cauchy inequality, we have the estimate:
$$
\left| \int_{h\le|x|<2h}u\nabla
				u\nabla\eta_h dx\right|\le C\int_{h\le|x|<2h}|\nabla u|\frac{1}{h} dx\le C\left(\int_{h\le|x|<2h}|\nabla u|^2 dx\right)^{\frac12}\rightarrow0\text{ for } h\rightarrow+\infty,
				$$
				and the last convergence is motivated by $|\nabla u|\in L^2(\R^2)$.

Consequently, from \eqref{e4.2} passing to the limit for $h\rightarrow
+\infty$, we find
\begin{equation}
\label{E2}\int_{\R^{2}}|u|^{2}j\left( 1-|u|^{2}\right) dx =\int
_{\R^{2}}|\nabla u|^{2} dx<+\infty.
\end{equation}
Let us define the set $B=\left\{  x\in\R^{2}:z_{1}\leq|u(x)|\leq
z_{2}\right\}  $ with $z_{1}$ and $z_{2}$ such that
\begin{equation}
1-{\dfrac{\eta_{0}}{4}}<z_{1}<z_{2}<1-\dfrac{\eta_{0}}{8},\label{k1}%
\end{equation}
where $\eta_{0}$ is given in $(\mathit{H4})$.

We claim that $B$ is bounded. Indeed, arguing by contradiction, let us suppose
that there exists a sequence $(x_{k})_{k}\subset B$ such that $\|x_{k}\|$ goes to $+\infty$ as $k$ goes to $+\infty$.

Let us fix $R_{0}=\dfrac{\dfrac{\eta_{0}}{4}+z_{1}-1}{M}$ where $\|\nabla
u\|_{\infty}=M$. Then, {by mean value theorem}, for every $x\in B_{R_{0}}\left(  x_{k}\right)$, we get
\begin{equation}\label{l1}
|u(x)|\leq M|x_{k}-x|+|u(x_{k})|\leq\dfrac{\eta_{0}}{4}+z_{1}-1+z_{2}
\end{equation}
and
\begin{equation}\label{l2}
|u(x)|\geq|u(x_{k})|-M|x_{k}-x|\geq z_{1}-\dfrac{\eta_{0}}{4}-z_{1}%
+1=1-\dfrac{\eta_{0}}{4}.
\end{equation}
 Moreover, for every
$x\in B_{R_{0}}\left(  x_{k}\right)   $ by Proposition \ref{mod1} and (\ref{l2}) we have
\begin{equation}\label{l3}
1-|u(x)|^{2}=(1-|u(x)|)(1+|u(x)|)\leq2\dfrac{\eta_{0}}{4}=\dfrac{\eta_{0}}{2}.
\end{equation}
By Proposition \ref{mod1}, $(\ref{k1})$ and $(\ref{l1})$ we have
\begin{equation}\label{l4}
1-|u(x)|^{2}\geq(1-|u(x)|)\geq\left(  2-\dfrac{\eta_{0}}{4}-z_{1}-z_{2}\right)  >0.
\end{equation}
Hence, by using $(\ref{l2})$, $(\ref{l3})$, assumption $(\mathit{H4})$ and $(\ref{l4})$, we get
\begin{equation}
\int_{B_{R_{0}}\left(  x_{k}\right)   }|u|^{2}j\left(  1-|u|^{2}\right) dx
\geq|B_{R_{0}}\left(  x_{k}\right) |\left(  1-\dfrac{\eta_{0}}{4}\right)
^{2}j\left[  \left(  2-\dfrac{\eta_{0}}{4}-z_{1}-z_{2}\right)  \right]
.\label{aa}
\end{equation}

By $(\ref{E2})$, for $k$ large enough, there exists $R$ depending only on $M$,
$\eta_{0}$, $z_{1}$ and $z_{2}$ such that
\begin{equation}\label{bb}
\displaystyle\int_{B_{R_{0}}\left(  x_{k}\right)}|u|^{2}j\left(  1-|u|^{2}\right) dx
\leq \int_{|x|>R}|u|^{2}j\left(  1-|u|^{2}\right) dx  
\end{equation}
 Since we know by \eqref{E2} that $\displaystyle\int_{\R^2}|u|^{2}j\left(  1-|u|^{2}\right) dx<\infty$ we conclude that
 $$
 \lim_{\|x_k\|\rightarrow +\infty}\displaystyle\int_{B_{R_{0}}\left(  x_{k}\right)}|u|^{2}j\left(  1-|u|^{2}\right) dx=0
 $$
 and this contradicts 
 $(\ref{aa})$.

\textbf{Step \textcolor{black}{2.}} As the set $B$ is bounded, by $(\ref{E2})$ as in \cite{BMR}
we get either
\begin{equation}
\int_{\R^{\textcolor{black}{2}}}|u|^{2} dx<+\infty\label{B2}%
\end{equation} or
\begin{equation}
\int_{\R^{\textcolor{black}{2}}}j\left(  1-|u|^{2}\right) dx  <+\infty.\label{C2}%
\end{equation}
When $(\ref{C2})$ holds, by  $(\mathit{H1})$, $(\mathit{H4})$, Proposition \ref{mod1} and $(\ref{c1})$ we can apply the
result of \cite{A} and conclude. 

\textcolor{black}{More precisely, as the set B is bounded, it means that for some $R>0$ big enough one has for $|x|>R$ either $|u|>1-\eta_0/8$ or $|u|<1-\eta_0/4$. If we are in this situation we have $J(t)<tj(t)$, by convexity property ($H4$), hence $\displaystyle \int_{\R^2} J(1-|u|^2) dx<\infty$ and then we can apply the result in \cite{A}.}\\
Then, let us suppose that $(\ref{B2})$ holds. \textcolor{black}{Acting as in \cite{BMR}, we would multiply equation in problem \eqref{p1} by $\left(  \sum x_{i}%
\dfrac{\partial u}{\partial x_{i}}\right)  \eta_{h}$, where $\eta_{h}$ is as in \textbf{Step 1.}. It is easy to see that
$$
\int_{\R^{\textcolor{black}{2}}}\left( \Delta u\right)  \left(  \sum x_{i}%
\dfrac{\partial u}{\partial x_{i}}\right)  \eta_{h} dx\longrightarrow 0,%
$$
hence}
\begin{equation}
\int_{\R^{\textcolor{black}{2}}}uj\left(  1-|u|^{2}\right)  \left(  \sum x_{i}%
\dfrac{\partial u}{\partial x_{i}}\right)  \eta_{h} dx\longrightarrow0\label{K2}%
\end{equation}
when $h$ goes to $+\infty$. \newline On the other hand we observe that
\begin{equation}
\begin{array}{l}
\displaystyle\int_{\R^{\textcolor{black}{2}}}uj\left(  1-|u|^{2}\right)  \left(  \sum
x_{i}\dfrac{\partial u}{\partial x_{i}}\right)  \eta_{h} dx=-\dfrac{1}
{2}\displaystyle\int_{\R^{\textcolor{black}{2}}}\sum x_{i}\left(\dfrac{\partial}{x_{i}}J\left(
1-|u|^{2}\right)\right)  \eta_{h} dx\\
\\
=\dfrac{1}{2}\displaystyle\int_{\R^{\textcolor{black}{2}}}\sum x_{i}\dfrac{\partial}
{x_{i}}\left[  J(1)-J\left(  1-|u|^{2}\right)  \right]  \eta_{h} dx
=-\int_{\R^{\textcolor{black}{2}}}\left[  J(1)-J\left(  1-|u|^{2}\right)  \right]
\eta_{h} dx\\
\\
-\dfrac{1}{2}\displaystyle\int_{\R^{\textcolor{black}{2}}}\left[  J(1)-J\left(
1-|u|^{2}\right)  \right]  x\cdot\nabla\eta_{h} dx.
\end{array}
\label{C3}%
\end{equation}
By the mean value theorem  applied to function $J$, there exists $\vartheta
\in(0,1)$ such that
\[
J(1)-J\left(  1-|u|^{2}\right)  =j\left(  1-\vartheta|u|^{2}\right)  |u|^{2}.%
\]
Then $\left(  \ref{C3}\right)  $ becomes
\begin{equation}%
\begin{array}{l}%
\displaystyle\int_{\R^{\textcolor{black}{2}}}uj\left(  1-|u|^{2}\right)  \left(  \sum
x_{i}\dfrac{\partial u}{\partial x_{i}}\right)  \eta_{h} dx+\dfrac{1}%
{2}\displaystyle\int_{\R^{\textcolor{black}{2}}}j\left(  1-\vartheta|u|^{2}\right)
|u|^{2}x\cdot\nabla\eta_{h} dx\\
\\
=-\displaystyle\int_{\R^{\textcolor{black}{2}}}j\left(  1-\vartheta|u|^{2}\right)
|u|^{2}\eta_{h} dx.
\end{array}
\label{C3*}%
\end{equation}
By Proposition \ref{mod1} we have $\left(  1-\vartheta|u|^{2}\right)
\in(0,1)$, then by the regularity of function $j$ and properties of function
$\eta_{h}$, acting as in \cite{BMR}, by (\ref{B2}), we easly get
\begin{equation}
\int_{\R^{\textcolor{black}{2}}}j\left(  1-\vartheta|u|^{2}\right)  |u|^{2}x\cdot
\nabla\eta_{h} dx\longrightarrow0\label{C4}%
\end{equation}
when $h$ goes to $+\infty$. By $\left(  \ref{K2}\right)  $, $\left(
\ref{C3*}\right)  $, $\left(  \ref{C4}\right)  $ we obtain
\[
-\int_{\R^{\textcolor{black}{2}}}j\left(  1-\vartheta|u|^{2}\right)  |u|^{2}\eta
_{h} dx\longrightarrow-\int_{\R^{\textcolor{black}{2}}}j\left(  1-\vartheta|u|^{2}\right)
|u|^{2} dx=0
\]
as $h$ goes to $+\infty,$ which \textcolor{black}{by $(\mathit{H2})$} directly implies $u=0$. Then the theorem is
completely proved.

\section{Proof of Theorem $\ref{teo2} $}

Throughout this section, $u$ will be a smooth solution of problem \eqref{p1} satisfying
\begin{equation}
\label{pot1}\int_{\R^{2}}J\left( 1-\left|  u\right| ^{2} \right) dx
<\infty.
\end{equation}
Let us prove some results which will be useful in the sequel.

\begin{proposition}
[Pohozaev identity]\label{prop2.1} Let $u$ be a smooth solution of problem
\eqref{p1}. Then for every $r>0$, it holds
\begin{equation}
\label{DHP40}\int_{S_{r}}\left|  \dfrac{\partial u}{\partial\nu}\right|
^{2}\textcolor{black}{d\sigma}+\dfrac{2}{r}\int_{B_{r}}J\left(  1-|u|^{2}\right)\textcolor{black}{dx} =\int_{S_{r}}\left|
\dfrac{\partial u}{\partial\tau}\right| ^{2}\textcolor{black}{d\sigma}+\int_{S_{r}}J\left(
1-|u|^{2}\right)\textcolor{black}{d\sigma} .
\end{equation}

\end{proposition}

\begin{proof}
Multiplying \eqref{p1} by $(x\cdot \nabla u)$ and integrating over $B_r$, we have
\begin{equation}
\begin{array}{l}
\displaystyle\int_{B_r}(x\cdot\nabla u)\Delta u\,\textcolor{black}{dx}
=
\displaystyle\int_{S_r}\dfrac{\partial u}{\partial\nu}(x\cdot\nabla
u)\textcolor{black}{d\sigma}-
\displaystyle\int_{B_r}\nabla(x\cdot\nabla
u)\nabla u\,\textcolor{black}{dx}\\
\\
=
\displaystyle\int_{S_r}(x\cdot \nu)
\left| \dfrac{\partial u}{\partial\nu}\right|^{2}\textcolor{black}{d\sigma}-\frac{1}{2}\int_{B_r}x\cdot\nabla
(|\nabla u|^2)\textcolor{black}{dx}-
\displaystyle\int_{B_r}|\nabla u|^2\textcolor{black}{dx}\\
\\=
\displaystyle\int_{S_r}(x\cdot
\nu)\left| \dfrac{\partial u}{\partial\nu}\right|^2\textcolor{black}{d\sigma}-\frac{1}{2}\int_{S_r}(x\cdot \nu) |\nabla u|^2\textcolor{black}{d\sigma}
\end{array} \label{e2.2}
\end{equation}
and
\begin{equation}
\begin{array}{l}
\displaystyle\int_{B_r}(x\cdot\nabla u)uj(1-|u|^2)\textcolor{black}{dx}
=-\dfrac{1}{2}
\displaystyle\int_{B_r}(x\cdot\nabla J(1-|u|^2))\textcolor{black}{dx}\\
\\=-\dfrac{1}{2}\left(-
\displaystyle\int_{B_r}2J(1-|u|^2)\textcolor{black}{dx}+
\displaystyle\int_{S_r}(x\cdot \nu) J(1-|u|^2)\textcolor{black}{d\sigma}\right)\\
\\=
\displaystyle\int_{B_r}J(1-|u|^2)\textcolor{black}{dx}-\frac{r}{2}\int_{S_r}J(1-|u|^2)\textcolor{black}{d\sigma}.
\end{array} \label{e2.3}
\end{equation}
Moreover, since $|\nabla u|^2=\left| \dfrac{\partial u}{\partial\nu}\right|^{2}+\left| \dfrac{\partial u}{\partial\tau}\right|^{2}$, by $\left(\ref{e2.2} \right)$ and $\left(\ref{e2.3} \right)$ we have
\begin{equation}\label{e2.90}
-\dfrac{r}{2}\int_{S_r}\left| \dfrac{\partial u}{\partial\nu}\right|^{2}\textcolor{black}{d\sigma}+\frac{r}{2}\int_{S_r} \left| \dfrac{\partial u}{\partial\tau}\right|^{2}\textcolor{black}{d\sigma}=\int_{B_r}J(1-|u|^2)\textcolor{black}{dx}-\frac{r}{2}\int_{S_r}J(1-|u|^2)\textcolor{black}{d\sigma}.
\end{equation}
Finally, by \eqref{e2.2}, \eqref{e2.3} and \eqref{e2.90}, we obtain \eqref{DHP40}.
\end{proof}
Now we prove the following result.

\begin{proposition}
\label{thm1.1} Assume $(\mathit{H1})$, $(\mathit{H2})$ and $\left( \ref{pot1}\right) $. Let $u$ be a smooth solution of
problem \eqref{p1}. Then%

\begin{equation}
\label{DHP2}\|\nabla u\|_{L^{\infty}(\R^{2})}< +\infty,
\end{equation}

\begin{equation}
|u(x)|\to1,\quad as \quad|x|\to\infty.\label{DHP3}%
\end{equation}
and
\begin{equation}
\label{DHP4}\int_{B_{R}}\left| \nabla u\right| ^{2} dx\leq CR,
\end{equation}
\noindent where $C$ is a positive constant independent of $R$.
\end{proposition}

\begin{proof}
First of all we prove $\left(\ref{DHP2} \right)$.\\
To this aim, for any $y\in\R^{2}$, let us denote by $B_{2}\left(y\right)$ the ball of radius 2 of $\R^{2}$ and consider the following \textcolor{black}{equation}
\begin{equation}\label{p1*}
\begin{array}{cc}
-\Delta u=u j(1-|u|^{2}) & $in$\,  B_2\left(y\right). \\
\end{array}
\end{equation}
By interior $W^{2,p}$ estimates, with $p > 2$ for the solution of the equation $\left(\ref{p1*} \right)$, by Proposition \ref{mod1} and by regularity of $j$, there exits a constant $C$ independent of $y$ such that
\begin{align*}
\|u\|_{W^{2,p}
(B_{1}\left(y\right))}\leq& C \left( \|u\|_{L^{\infty}
(B_{2}\left(y\right))}+\|u j(1-|u|^{2})\|_{L^{p}
(B_{2}\left(y\right))}\right)\le C.
\end{align*}
Finally, by using the Sobolev embedding for $p > 2$, we have $W^{2,p}(B_{1}\left(y\right))\subset C^{1}(\overline{B}_{1}\left(y\right))$, hence  $\left(\ref{DHP2} \right)$.\\
Let us prove $\left(\ref{DHP3} \right)$. To this aim we suppose that it were not true, hence there exists a sequence $|x_{n}|\rightarrow \infty$ such that $\left|u\left( x_{n}\right) \right|\leq 1-\delta$ for some $\delta >0$. Let us consider the ball $B_{\frac{\delta}{2M}}\left(x_{n}\right)$ where $M=\|\nabla u\|_{L^{\infty}}$ as $u\in C^{1}\left( \R^{2}\right)$. Then, by mean value theorem, we get
\begin{equation*}
\left|u(x)-u\left( x_{n}\right) \right|\leq M| x- x_{n}|\leq M\frac{\delta}{2M}=\frac{\delta}{2}.
\end{equation*}
Then
\begin{equation}\label{DHP100}
\left|u(x)\right|=\left|u(x)-u\left( x_{n}\right) \right|+\left|u\left( x_{n}\right) \right|\leq \frac{\delta}{2}+1-\delta= 1-\frac{\delta}{2}.
\end{equation}
By Proposition \ref{mod1},  $\left|u(x)\right|^{2}\leq \left|u(x)\right|$ and by $\left(\ref{DHP100} \right)$, we get $1-\left|u(x)\right|^{2}\geq \dfrac{\delta}{2}$. By $(\mathit{H2})$, function $J$ is increasing on $(0,1]$ hence $J\left( 1-\left|u(x)\right|^{2}\right)\geq J\left(\dfrac{\delta}{2} \right)$ and
\begin{equation}\label{DHP501}
\int_{B_{\frac{\delta}{2M}}\left(x_{n}\right)}J\left( 1-\left|u(x)\right|^{2}\right) dx\geq J\left(\frac{\delta}{2}\right)\pi\frac{\delta^{2}}{4M^{2}}.
\end{equation}
By $\left(\ref{pot1} \right)$  there exists $R_{0}$ such that
\begin{equation}\label{DHP501*}
\int_{\left|x\right|>R_{0}}J\left( 1-\left|u(x)\right|^{2}\right) dx< J\left(\frac{\delta}{2}\right)\pi\frac{\delta^{2}}{4M^{2}}.
\end{equation}
As $\|x_{n}\|\rightarrow \infty$  there exists a ball $B_{\frac{\delta}{2M}}\left(x_{n}\right)\subset \R^{2}\setminus B_{R_{0}}$, then $\left(\ref{DHP501} \right)$ and $\left(\ref{DHP501*} \right)$ are clearly in contradiction.\\
Finally, in order to prove $\left(\ref{DHP4} \right)$, let us multiply $\left(\ref{p1} \right)$ by $u$ and integrate over $B_{R}$. We obtain
\begin{equation*}
\int_{B_{R}}\left|\nabla u\right|^{2} dx=\int_{S_{R}}\dfrac{\partial u}{\partial \nu}u \, \textcolor{black}{d\sigma} + \int_{B_{R}}\left| u\right|^{2}j\left(1-\left| u\right|^{2} \right) dx,
\end{equation*}
\noindent where $\nu$ denotes the outward normal to $B_{R}$.
It is easy to show that
\begin{equation}\label{der}
\left|\int_{S_{R}}\dfrac{\partial u}{\partial \nu}u\, \textcolor{black}{d\sigma}\right|\leq 2M\pi R,
\end{equation}
then it remains to prove
\begin{equation}\label{DHP5}
\int_{B_{R}}\left| u\right|^{2}j\left(1-\left| u\right|^{2} \right) dx \leq CR
\end{equation}
for some constant $C$ independent of $R$. For this purpose let us observe that by Proposition \ref{mod1} and Cauchy-Schwarz inequality we have
\begin{equation}\label{CS}
\int_{B_{R}}\left| u\right|^{2}j\left(1-\left| u\right|^{2} \right) dx\leq \sqrt{\pi}R\left(\int_{B_{R}} J'^{2}\left(1-\left| u\right|^{2} \right)dx\right)^{\frac{1}{2}}.
\end{equation}
Now we observe that
\begin{equation}\label{DHP6}
J'^{2}(t)\leq C J(t)\,\,\,\forall t\in \left[0,1 \right].
\end{equation}
Indeed as $J''$ is a continuous function, there exists a positive constant $M$ such that
\begin{equation*}
J''(t)\leq M \,\forall t\in \left[0,1 \right].
\end{equation*}
Multiplying both members of previous inequality by $J'(t)$ which is nonnegative by $(\mathit{H2})$, we obtain
\begin{equation*}
J''(t)J'(t)\leq M J'(t) \,\forall t\in \left[0,1 \right]
\end{equation*}
which is
\begin{equation*}
\frac{1}{2}\frac{d}{dt}\left( J'(t)\right) ^{2}\leq M J'(t) \,\,\,\forall t\in \left[0,1 \right].
\end{equation*}
Taking into account that $J(0)=J'(0)=0$  and integrating both members we get
\begin{equation*}
\left( J'(t)\right) ^{2}\leq 2M J(t)\,\forall t\in \left[0,1 \right],
\end{equation*}
hence $\left(\ref{DHP6} \right)$ holds. By (\ref{der}), (\ref{CS}), (\ref{DHP6}) and Proposition \ref{mod1} we obtain $\left(\ref{DHP4} \right)$.
\end{proof}
By previous results, $\deg(u, S_{R})$ is well defined for $R$ large (see \cite{BMR}) and we denote
$d=|\deg(u,S_{R})|$. By virtue of \eqref{DHP3}, there exists
$R_{0}>0$, such that
\begin{equation}
\label{DHP8}|u(x)|\geq \dfrac{3}{4}, \quad\text{for } |x|=R\geq R_{0}
\end{equation}
and a smooth single-valued function $\psi(x)$, defined for $|x|\geq
R_{0}$, such that%

\begin{equation}
\label{DHP9}u(x)=\rho(x)e^{i(d\theta+\psi(x))},
\end{equation}
where $\rho=|u|$. If we denote
\begin{equation}\label{fi}
\varphi(x)=d\theta(x)+\psi(x),
\end{equation}
then $\varphi(x)$ is
well defined and smooth locally on the set $|x|\geq R_{0}$.

Now we are able to prove the following result:

\begin{proposition}
\label{Prop1} Assume $(\mathit{H2})$ and $(\mathit{H3})$. Let $u$ be a smooth solution of problem \eqref{p1} written as in $\left(
\ref{DHP9} \right) $. Then
\begin{equation}
\label{DHP7*}\lim_{R\rightarrow+\infty}\frac{1}{\log R}\int_{B_{R}\setminus
B_{R_{0}}}\left|  \nabla\psi\right| ^{2} \textcolor{black}{ dx}=0
\end{equation}
and
\begin{equation}
\label{DHP7}\lim_{R\rightarrow+\infty}\frac{1}{\log R}\int_{B_{R}\setminus
B_{R_{0}}}\left|  \nabla\rho\right| ^{2}\textcolor{black}{ dx}=0
\end{equation}

\end{proposition}

\begin{proof}
By putting $\left(\ref{DHP9} \right)$ in $\left(\ref{p1} \right)$ we get
\begin{equation*}
-\Delta u=-\left( \Delta \rho\right)e^{i\varphi}-2i\left(\nabla \rho\nabla\varphi \right)e^{i\varphi}-\rho e^{i\varphi}\left( i\Delta\varphi-\left| \nabla \varphi\right|^{2}\right)=\rho e^{i\varphi}j\left( 1-\rho^{2}\right).
\end{equation*}
Separating the real and imaginary parts we obtain
\begin{equation}\label{DHP8*}
\rho\Delta\varphi+2\nabla\rho\nabla\varphi=0 \,\,\,\, \hbox{for}\, \left| x\right|\geq R_{0}
\end{equation}
\begin{equation}\label{DHP9*}
-\Delta\rho+\rho\left| \nabla \varphi\right|^{2}=\rho j\left( 1-\rho^{2}\right) \,\,\,\, \hbox{for}\, \left| x\right|\geq R_{0}.
\end{equation}
Let us observe that equation $\left(\ref{DHP8*} \right)$ can be written as
\begin{equation}\label{DHP11}
\hbox{div}\left(\rho^{2}\nabla\varphi \right) =0 \,\,\,\, \hbox{for}\,  \left| x\right|\geq R_{0}.
\end{equation}
Let $V(x)$ be the vector field in $\R^{2}\setminus\lbrace 0 \rbrace$ defined by
\begin{equation*}
V\left( r \cos\theta,r \sin \theta\right)=\left( -\sin \theta, \cos \theta\right).
\end{equation*}
By (\ref{fi}) we have
\begin{equation}\label{DHP10}
\nabla\varphi= d\nabla\theta+\nabla\psi=\frac{d}{r}V+\nabla\psi.
\end{equation}
Then, combining $\left(\ref{DHP11} \right)$ and $\left(\ref{DHP10} \right)$ we have
\begin{equation}\label{DHP12}
\hbox{div} \left( \rho^{2}\left(\frac{d}{r}V+\nabla\varphi \right) \right)=0 \,\,\,\, \hbox{for}\, \left| x\right|\geq R_{0}.
\end{equation}
Now we want to prove that
\begin{equation}\label{DHP12*}
\int_{S_{R}}\rho^{2}\frac{\partial\psi}{\partial\nu}\textcolor{black}{ d\sigma}=0.
\end{equation}
To this aim let us consider the vector field $D=(u\wedge u_{x_{1}},u\wedge u_{x_{2}})$ which is well defined and smooth on $\R^{2}$. Note that by equation
$\left(\ref{p1} \right)$
\begin{equation}\label{DHP13}
\hbox{div} D=u\wedge \Delta u=0
\end{equation}
and by integrating $\left(\ref{DHP13} \right)$ over $B_{R}$ we have
\begin{equation}\label{DHP14}
\int_{S_{R}}D\cdot\nu\textcolor{black}{ d\sigma} =0 \,\,\, \forall R\geq R_{0}.
\end{equation}
On the other hand, we have
\begin{equation}
D=\rho^{2}\nabla\varphi \,\,\,\,\hbox{for} \, \left| x\right|\geq R_{0},
\end{equation}
so $\left(\ref{DHP10} \right)$  and the fact that $V\cdot \nu =0$ on $S_{R}$ give $\left(\ref{DHP12*} \right)$.
\\
Now, we want to prove $\left(\ref{DHP7*} \right)$.
To this aim let us pose
\begin{equation*}
\psi_{R}=\frac{1}{2\pi R}\int_{S_{R}}\psi \textcolor{black}{ d\sigma}.
\end{equation*}
Multiplying $\left(\ref{DHP11} \right)$ by $\psi-\psi_{R}$ and integrating over $A_{R}=B_{R}\setminus B_{R_{0}}$, we obtain
\begin{equation}\label{DHP16}
\begin{array}{l}
\displaystyle\int_{A_{R}}\rho^{2}\left( \dfrac{d}{r}V+\nabla\psi\right)\nabla\psi\textcolor{black}{ dx}=\int_{S_{R}}\rho^{2}\left( \dfrac{d}{r}V\cdot\nu +\dfrac{\partial\psi}{\partial\nu}\right)\left( \psi-\psi_{R}\right)\textcolor{black}{ d\sigma}\\
\\
\displaystyle -\int_{S_{R_{0}}}\rho^{2}\left( \dfrac{d}{r}V\cdot\nu+\dfrac{\partial\psi}{\partial\nu}\right)\left( \psi-\psi_{R}\right)\textcolor{black}{ d\sigma}.\\
\end{array}
\end{equation}
As $V\cdot\nu=0$, by $\left(\ref{DHP12*} \right)$ we get
\begin{equation}\label{DHP17}
\int_{S_{R_{0}}}\rho^{2}\left( \dfrac{d}{r}V\cdot\nu+\dfrac{\partial\psi}{\partial\nu}\right)\left( \psi-\psi_{R}\right)\textcolor{black}{ d\sigma}=\int_{S_{R_{0}}}\rho^{2} \frac{\partial\psi}{\partial\nu}\psi\textcolor{black}{ d\sigma}=C
\end{equation}
where $C$ is a constant independent of $R$. We also observe that
\begin{equation}\label{DHP18}
\int_{A_{R}}\dfrac{d}{r}V\cdot\nabla\psi \textcolor{black}{ dx}=\int_{A_{R}}\frac{d}{r}\dfrac{\partial\psi}{\partial\tau}\textcolor{black}{ dx}=0.
\end{equation}
By $\left(\ref{DHP16} \right)$, $\left(\ref{DHP17} \right)$ and $\left(\ref{DHP18} \right)$ we have
\begin{equation}\label{DHP19}
\int_{A_{R}}\rho^{2}\left| \nabla\psi\right|^{2}\textcolor{black}{ dx}\leq\int_{S_{R}}\left| \dfrac{\partial\psi}{\partial\nu}\right|\left| \psi-\psi_{R}\right|\textcolor{black}{ d\sigma}+\int_{A_{R}}\left(1-\rho^{2} \right)\dfrac{d}{r}\left| \nabla\psi\right|\textcolor{black}{ d\sigma}+C.
\end{equation}
Cauchy-Schwarz inequality implies
\begin{equation*}
\left|\int_{S_{R}} \dfrac{\partial\psi}{\partial\nu} \left( \psi-\psi_{R}\right)\textcolor{black}{ d\sigma} \right|\leq\left(\int_{S_{R}}\left| \dfrac{\partial\psi}{\partial\nu}\right|^{2}\textcolor{black}{ d\sigma} \right)^{\frac{1}{2}}\left(\int_{S_{R}}\left| \psi-\psi_{R}\right|^{2} \textcolor{black}{ d\sigma}\right)^{\frac{1}{2}}.
\end{equation*}
We recall the following Poincar\'e inequality  
\begin{equation}\label{DHP20}
\int_{S_{R}}\left| \psi-\psi_{R}\right|^{2}\textcolor{black}{ d\sigma}\leq R^{2}\int_{S_{R}}\left| \nabla_{\tau}\psi\right|^{2}\textcolor{black}{ d\sigma}.
\end{equation}
Therefore by $\left(\ref{DHP20} \right)$ and Young inequality we have
\begin{equation}\label{DHP21}
\left|\int_{S_{R}} \dfrac{\partial\psi}{\partial\nu} \left( \psi-\psi_{R}\right)\textcolor{black}{ d\sigma} \right|\leq \dfrac{R}{2}\int_{S_{R}}\left| \dfrac{\partial\psi}{\partial\nu}\right|^{2}\textcolor{black}{ d\sigma}+\dfrac{R}{2}\int_{S_{R}}\left| \nabla_{\tau}\psi\right|^{2}\textcolor{black}{ d\sigma}= \dfrac{R}{2}\int_{S_{R}}\left| \nabla\psi\right|^{2}\textcolor{black}{ d\sigma}.
\end{equation}
By $\left(\ref{DHP21} \right)$ and $\left(\ref{DHP19} \right)$ we obtain
\begin{equation}\label{DHP22}
\int_{A_{R}}\rho^{2}\left| \nabla\psi\right|^{2} \textcolor{black}{ dx}\leq\dfrac{R}{2}\int_{S_{R}}\left| \nabla\psi\right|^{2}\textcolor{black}{ d\sigma}+\int_{A_{R}}\left(1-\rho^{2} \right)\dfrac{d}{r}\left| \nabla\psi\right|\textcolor{black}{ dx}+C.
\end{equation}

Remember that $A_R=B_R\setminus B_{R_0}$. Given $0<\epsilon<\frac14$ we choose $R_0$ big enough such that 
$$
1-\epsilon\le \rho^2(x)\le 1,\, \mbox{ for }|x|\ge R_0.
$$
Denote by 
\begin{equation}
\label{def_f}
f(R)=\int_{A_{R}}\left| \nabla\psi\right|^{2}\textcolor{black}{ dx}.
\end{equation}
whose derivative is 
$$
f'(R)=\int_{S_{R}}\left| \nabla\psi\right|^{2}\textcolor{black}{ d\sigma}
$$
and by \eqref{DHP4} and \eqref{DHP10} satisfies the estimate
\begin{equation}
\label{est4}
f(R)\le CR.
\end{equation}
Thus \eqref{DHP22} becomes:
\begin{equation}\label{22+}
(1-\epsilon)f(R)\le\frac{R}{2}f'(R)+\int_{A_{R}}\left(1-\rho^{2} \right)\dfrac{d}{r}\left| \nabla\psi\right|\textcolor{black}{ dx}+C.
\end{equation}
We estimate the second term in \eqref{22+}. By   Young inequality we obtain
\begin{equation}\label{DHP23+}
\int_{A_{R}}\left(1-\rho^{2} \right)\dfrac{d}{r}\left| \nabla\psi\right|\textcolor{black}{ dx}\leq 
\frac{d^2}{\epsilon} \int_{A_{R}}\dfrac{\left( 1-\rho^{2}\right)^2}{r^{2}}\textcolor{black}{ dx}+ \epsilon \int_{A_{R}}\left| \nabla\psi\right|^{2}\textcolor{black}{ dx}=\frac{d^2}{\epsilon} \int_{A_{R}}\dfrac{\left( 1-\rho^{2}\right)^2}{r^{2}}\textcolor{black}{ dx}+\epsilon f(R) .
\end{equation}
and thus \eqref{22+} implies
\begin{equation}
\label{22++}
(1-2\epsilon)f(R)\le\frac{R}{2}f'(R)+\frac{d^2}{\epsilon} \int_{A_{R}}\dfrac{\left( 1-\rho^{2}\right)^2}{r^{2}}\textcolor{black}{ dx}+C.
\end{equation}
Observe, by l'Hospital rule that 

\begin{equation}\label{Hospital}
\lim_{R\rightarrow\infty}\frac{1}{\log R}\int_{A_{R}}\dfrac{\left( 1-\rho^{2}\right)^2}{r^{2}}\textcolor{black}{ dx}=\lim_{R\rightarrow\infty}
\frac{1}{R}\int_{S_{R}}( 1-\rho^{2})^2\textcolor{black}{ d\sigma}=0,
\end{equation}

because $\lim_{|x|\rightarrow\infty}\rho(x)=1$. Consequently we may rewrite \eqref{22++} in the form
\begin{equation}
\label{22+++}
f(R)\le\frac{R}{\beta}f'(R)+F(R),
\end{equation}
where $\beta=2(1-2\epsilon)>1$ \textcolor{black}{ and $F(R)=\displaystyle \frac{d^2}{ \epsilon(1-2\epsilon)}\int_{A_{R}}\frac{\left( 1-\rho^{2}\right)^2}{r^{2}}\textcolor{black}{ dx}+\dfrac{C}{1-2\epsilon}$ which, by \eqref{Hospital} satisfies}
\begin{equation}\label{F}
\lim_{R\rightarrow\infty}\frac{F(R)}{\log R}=0.
\end{equation}
\textcolor{black}{Denoted} by 
$$
g(R)=\frac{f(R)}{\log R},
$$
we want to prove 
$$
\lim_{R\rightarrow\infty}g(R)=0.
$$
We plug the expression for $g$ in \eqref{22+++} and obtain:
\begin{equation}\label{g+}
g(R)\le g'(R)\frac{R\log R}{\beta\log R-1}+\frac{\beta F(R)}{\beta\log R-1}=\frac{g'(R)}{\alpha(R)}+H(R),
\end{equation}
where we denoted by $\displaystyle \alpha(R)=\frac{\beta\log R-1}{R\log R}$ and by $\displaystyle H(R)=\frac{\beta F(R)}{\beta\log R-1}$.
Denote by $\displaystyle\tilde\beta=\frac{\beta+1}{2}>1$ and observe that 
\begin{equation}
\label{alpha}\frac{\tilde\beta}{R}\le\alpha(R)\le\frac{\beta}{R},
\end{equation}
for $R>\tilde R$ big enough. 

Observe that, by \textcolor{black}{\eqref{F}},
$
\lim_{R\rightarrow\infty}H(R)=0.
$
\textcolor{black}{Hence,} given $\mu>0$ there exists $R(\mu)>\tilde R$ such that $H(R)<\mu$ for all $R>R(\mu)$. Consequently, by \eqref{g+} we have 
for $R>R(\mu)$
\begin{equation}
\label{g+-}
\alpha(R)(g(R)-\mu)\le g'(R)=(g(R)-\mu)'.
\end{equation}
Given $R>\theta>R(\mu)$ we multiply \eqref{g+-} by $\displaystyle\exp\left(-\int_\theta^R\alpha(\tau)d\tau\right)$ and find that  the function

$$
R\rightarrow(g(R)-\mu)\exp\left(-\int_\theta^R\alpha(\tau)d\tau\right)
$$
is increasing and thus, using also the linear growth of $f$ \textit{i.e} \eqref{est4}, we find
\begin{equation}\label{in1}
(g(\theta)-\mu)\le (g(R)-\mu)\exp\left(-\int_\theta^R\alpha(\tau)d\tau\right)\le \left(\frac{CR}{\log R}-\mu\right)\exp\left(-\int_\theta^R\alpha(\tau)d\tau\right).
\end{equation}
By \eqref{alpha} we have 
$$
\exp\left(-\int_\theta^R\alpha(\tau)d\tau\right)\leq \left(\frac{\theta}{R}\right)^{\tilde\beta},
$$
which plugged into \eqref{in1} gives 
 \begin{equation}\label{in2}
 (g(\theta)-\mu)\le  \left(\frac{CR}{\log R}+\mu\right)\left(\frac{\theta}{R}\right)^{\tilde\beta}\rightarrow 0 \;\text{ for } R\rightarrow +\infty.
 \end{equation}
 We found that for  $\theta > R(\mu)$ one has $0\le g(\theta)\le \mu$. As $\mu>0$ may  be chosen arbitrarily small, 
 we obtain that 
 $$
 \lim_{R\rightarrow\infty}g(R)=0,
 $$
 that is \eqref{DHP7*}.

Now we want to prove $\left(\ref{DHP7} \right)$. To this aim let us consider the function $\eta$ as in the proof of Theorem \ref {teo2} and set $\eta_R(x)=\eta\left(\dfrac{x}{R}\right)$. Multiplying \eqref{DHP9*} by
$(1-\varrho) \eta_{R}$ and integrating over $\R^{2}\setminus B_{R_{0}}$, we obtain
\begin{equation}\label{DHP34}
\begin{array}{l}
\displaystyle\int_{\R^{2}\setminus B_{R_{0}}}\left| \nabla\rho\right|^{2}\eta_{R} \textcolor{black}{dx}=\displaystyle\int_{\R^{2}\setminus B_{R_{0}}}\left(1-\rho \right)\nabla\rho\nabla\eta_{R}\textcolor{black}{dx}+\int_{\R^{2}\setminus B_{R_{0}}}\rho\left(1-\rho \right)\left| \nabla\varphi\right|^{2}\eta_{R}\textcolor{black}{dx}\\
\\
-\displaystyle\int_{S_{R_{0}}}\dfrac{\partial\rho}{\partial\nu}\left(1-\rho \right)\eta_{R}\textcolor{black}{d\sigma}-\displaystyle\int_{\R^{2}\setminus B_{R_{0}}}\rho\left(1-\rho \right)j\left(1-\rho^{2} \right)\eta_{R}\textcolor{black}{dx}.\\
\end{array}
\end{equation}
\textcolor{black}{Now we estimate each term in \eqref{DHP34}}. By definition of  $\eta_R$ and as $\left( 1-\rho^{2}\right)\leq 1$, we have
\begin{equation}\label{DHP35}
\begin{array}{l}
\displaystyle\int_{\R^{2}\setminus B_{R_{0}}}\left(1-\rho \right)\nabla\rho\nabla\eta_{R}\textcolor{black}{dx}
=-\frac{1}{2}\displaystyle\int_{\R^{2}\setminus B_{R_{0}}}\nabla\left( 1-\rho\right)^{2}\nabla\eta_{R}\textcolor{black}{dx}=\dfrac{1}{2}\displaystyle\int_{\R^{2}\setminus B_{R_{0}}}\left( 1-\rho\right)^{2}\Delta\eta_{R}\textcolor{black}{dx}\\
\\
=\dfrac{1}{2}\displaystyle\int_{B_{2R}\setminus B_{R_{0}}}\left( 1-\rho\right)^{2}\Delta\eta_{R}\textcolor{black}{dx}\leq\frac{C}{R^{2}}\displaystyle\int_{B_{2R}\setminus B_{R_{0}}}\left( 1-\rho\right)^{2}\textcolor{black}{dx}\leq C.\\
\end{array}
\end{equation}
Let us observe that by $\left(\ref{DHP10} \right)$ we can write
\begin{equation}\label{DHP36}
\left| \nabla\varphi\right|^{2}\leq 2\left( \dfrac{d^{2}}{r^{2}}+\left| \nabla\psi\right|^{2}\right).
\end{equation}
Then by definition of $\eta_R$,  $\left(\ref{DHP8}\right)$ and $\left(\ref{DHP36} \right)$, we obtain
\begin{equation}\label{DHP37}
\begin{array}{l}
\displaystyle\int_{\R^{2}\setminus B_{R_{0}}}\rho\left(1-\rho \right)\left| \nabla\varphi\right|^{2}\eta_{R}\textcolor{black}{dx}\leq 2\displaystyle\int_{B_{2R}\setminus B_{R_{0}}}\rho\left(1-\rho \right)\left( \dfrac{d^{2}}{r^{2}}+\left| \nabla\psi\right|^{2}\right)\eta_{R}\textcolor{black}{dx}\\
\\
\leq 2d^{2}\displaystyle\int_{B_{2R}\setminus B_{R_{0}}}\dfrac{\rho\left(1-\rho \right)}{r^{2}}\textcolor{black}{dx}+2\int_{B_{2R}\setminus B_{R_{0}}}\rho\left(1-\rho \right)\left| \nabla\psi\right|^{2}\textcolor{black}{dx}\\
\\
\leq 2d^{2}\displaystyle\int_{B_{2R}\setminus B_{R_{0}}}\dfrac{\rho\left(1-\rho^{2} \right)} {\left(1+\rho \right)r^{2}}\textcolor{black}{dx}+2\int_{B_{2R}\setminus B_{R_{0}}}\left| \nabla\psi\right|^{2}\textcolor{black}{dx}\\
\\
\leq C_{1}\displaystyle\int_{B_{2R}\setminus B_{R_{0}}}\dfrac{\left(1-\rho^{2} \right)}{r^{2}}\textcolor{black}{dx}+C_{2}\displaystyle\int_{B_{2R}\setminus B_{R_{0}}}\left| \nabla\psi\right|^{2}\textcolor{black}{dx},\\
\end{array}
\end{equation}
$C_{1}$ and $C_{2}$ being constants independent of $R$.

\textcolor{black}{By }\eqref{Hospital} we have that 
\begin{equation}\label{IlogR}
\lim_{R\rightarrow\infty}\frac{1}{\log R}\int_{B_{2R}\setminus B_{R_{0}}}\dfrac{\left(1-\rho^{2} \right)}{r^{2}}\textcolor{black}{dx}=0,
\end{equation}
which combined with \eqref{DHP7*} gives that
\begin{equation}
\label{ologR}
\lim_{R\rightarrow\infty}\frac{1}{\log R}\int_{\R^{2}\setminus B_{R_{0}}}\rho\left(1-\rho \right)\left| \nabla\varphi\right|^{2}\eta_{R}\textcolor{black}{dx}=0.
\end{equation}
Finally, by $(\mathit{H2})$, we can assert that
\begin{equation}\label{DHP38}
-\int_{\R^{2}\setminus B_{R_{0}}}\rho\left(1-\rho \right)j\left(1-\rho^{2} \right)\eta_{R}\textcolor{black}{dx}\leq 0.
\end{equation}
By putting $\left(\ref{DHP35} \right)$, $\left(\ref{ologR} \right)$ and $\left(\ref{DHP38} \right)$ in $\left(\ref{DHP34} \right)$ and as $\eta_{R}(x)=1$ in $B_{R}$, we find
\begin{equation}\label{DHP7bis}
\lim_{R\rightarrow\infty}\frac{1}{\log R}\int_{B_{R}\setminus B_{R_{0}}}\left| \nabla\rho\right|^{2}\textcolor{black}{dx}\leq\lim_{R\rightarrow\infty}\frac{1}{\log R}\int_{\R^{2}\setminus B_{R_{0}}}\left| \nabla\rho\right|^{2}\eta_{R}\textcolor{black}{dx}=0,
\end{equation}
which is \eqref{DHP7}. 
\end{proof}

\subsection{Quantization of the energy}

The final step is to prove
\begin{equation}
\label{DHP39}\int_{\R^{2}}J\left(  1-\rho^{2}\right) \textcolor{black}{dx}=\pi d^{2}.
\end{equation}
Let us consider
\begin{equation*}
\begin{array}{l}
E=\displaystyle\int_{\R^{2}}J\left( 1-\rho^{2}\right)\textcolor{black}{dx},\\
\\E(r)=\displaystyle\int_{B_{r}}J\left( 1-\rho^{2}\right)\textcolor{black}{dx}.
\end{array}
\end{equation*}
We have that
\begin{equation}
\label{DHP41}
\begin{array}{l}
\displaystyle\lim_{r\rightarrow+\infty}E(r)=E,\\ 
\\
 \displaystyle\lim_{R\rightarrow
+\infty}\dfrac{1}{\log R}\displaystyle \int_{0}^{R}\dfrac{E(r)}{r}dr=E.
\end{array}
\end{equation}
By integrating $\left( \ref{DHP40} \right) $ for $r\in(0,R)$ we get
\[
\int_{B_{R}}\left|  \dfrac{\partial u}{\partial\nu}\right| ^{2}\textcolor{black}{dx}+2\int_{0}%
^{R}\dfrac{E(r)}{r}dr=\int_{B_{R}}\left|  \dfrac{\partial u}{\partial\tau
}\right| ^{2}\textcolor{black}{dx}+E(R).
\]
By dividing previous equality by $\log R$, we obtain
\begin{equation}
\label{DHP42}\dfrac{1}{\log R}\int_{B_{R}}\left|  \dfrac{\partial u}%
{\partial\nu}\right| ^{2}\textcolor{black}{dx}+\dfrac{2}{\log R}\int_{0}^{R}\dfrac{E(r)}%
{r}dr=\dfrac{1}{\log R}\int_{B_{R}}\left|  \dfrac{\partial u}{\partial\tau
}\right| ^{2}\textcolor{black}{dx}+\dfrac{1}{\log R}E(R).
\end{equation}
We observe that for $r>R_{0}$
\begin{equation}
\label{DHP43}\left|  \dfrac{\partial u}{\partial\nu}\right| ^{2}=\left|
\dfrac{\partial\rho}{\partial\nu}\right| ^{2}+\rho^{2}\left|  \dfrac
{\partial\varphi}{\partial\nu}\right| ^{2}\leq\left| \nabla\rho\right|
^{2}+\left| \nabla\psi\right| ^{2}%
\end{equation}
and by (\ref{DHP10})
\begin{equation*}
\left|  \dfrac{\partial u}{\partial\tau}\right| ^{2}=\left|
\dfrac{\partial\rho}{\partial\tau}\right| ^{2}+\rho^{2}\left|  \dfrac
{\partial\varphi}{\partial\tau}\right| ^{2}=\left|
\dfrac{\partial\rho}{\partial\tau}\right| ^{2}+\rho^{2}\left(\dfrac{d}{r}+\dfrac
{\partial\psi}{\partial\tau}\right)^{2}.
\end{equation*}
Hence
\begin{equation}\label{DHP44}
\left|  \dfrac{\partial u}{\partial\tau}\right| ^{2}-\dfrac{d^{2}}{r^{2}}
= \left|  \dfrac{\partial\rho
}{\partial\tau}\right| ^{2}-(1-\rho^{2})\dfrac{d^{2}}{r^{2}}+\left|\dfrac{\partial\psi
}{\partial\tau}\right|^{2}-2(1-\rho^{2}) \dfrac{d}{r}\dfrac{\partial\psi
}{\partial\tau}+2\dfrac{d}{r}\dfrac{\partial\psi
}{\partial\tau}.
\end{equation}
By integrating on $B_{R}\setminus B_{R_{0}}$ and using (\ref{DHP18}) we easily get
\begin{equation}\label{DHP44bis}
\begin{array}{l}
\displaystyle\int_{B_{R}\setminus B_{R_{0}}}\left(\left|  \dfrac{\partial u}{\partial\tau}\right| ^{2}-\dfrac{d^{2}}{r^{2}}\right)\textcolor{black}{dx}
=\displaystyle\int_{B_{R}\setminus B_{R_{0}}} \left(\left|  \dfrac{\partial\rho
}{\partial\tau}\right| ^{2}+\left|\dfrac{\partial\psi
}{\partial\tau}\right|^{2}\right)\textcolor{black}{dx}-d^{2}\displaystyle\int_{B_{R}\setminus R_{0}}\dfrac{1-\rho^{2}}{r^{2}}\textcolor{black}{dx}\\
\\
-2d\displaystyle\int_{B_{R}\setminus R_{0}} \dfrac{1-\rho^{2}}{r}\dfrac{\partial\psi
}{\partial\tau}\textcolor{black}{dx}.
\end{array}
\end{equation}
By  Proposition \ref{mod1} and Cauchy-Schwarz inequality, (\ref{DHP44bis}) becomes
$$
\left|\displaystyle\int_{B_{R}\setminus B_{R_{0}}}\left(\left|  \dfrac{\partial u}{\partial\tau}\right| ^{2}-\dfrac{d^{2}}{r^{2}}\right)\textcolor{black}{dx}\right|\leq
$$
\begin{equation}\label{DHP45}
\leq \displaystyle\int_{B_{R}\setminus B_{R_{0}}}\left( \left| \nabla \rho\right| ^{2}+\left|\nabla \psi\right|^{2}\right)\textcolor{black}{dx}+d^{2}\displaystyle\int_{B_{R}\setminus R_{0}}\dfrac{1-\rho^{2}}{r^{2}}\textcolor{black}{dx}+2 d \sqrt{\log \dfrac{R}{R_{0}}}\left(\displaystyle\int_{B_{R}\setminus R_{0}} \left|\nabla \psi\right|^{2}\textcolor{black}{dx}\right)^{\frac{1}{2}}.
\end{equation}
Combining   $\left( \ref{DHP7*} \right) $,$\left( \ref{DHP7} \right) $, \eqref{IlogR},
 and $\left( \ref{DHP45} \right) $, we have
\[
\lim_{R\rightarrow+\infty}\left|\dfrac{1}{\log R}\int_{B_{R}\setminus B_{R_{0}}%
} \left(\left|  \dfrac{\partial u}{\partial\tau}\right| ^{2}-\dfrac{d^{2}%
}{r^{2}}\right)\textcolor{black}{dx}\right| =0,
\]
and so
\begin{equation}
\label{DHP46}\lim_{R\rightarrow+\infty}\dfrac{1}{\log R}\int_{B_{R}}\left|
\dfrac{\partial u}{\partial\tau}\right| ^{2}\textcolor{black}{dx}=2\pi d^{2}.
\end{equation}
Moreover, by $\left( \ref{DHP7*} \right) $, $\left( \ref{DHP7} \right) $ and $\left( \ref{DHP43} \right) $ it results
\begin{equation}
\label{DHP47}\lim_{R\rightarrow+\infty}\dfrac{1}{\log R}\int_{B_{R}\setminus
B_{R_{o}}}\left|  \dfrac{\partial u}{\partial\nu}\right| ^{2}\textcolor{black}{dx}=0.
\end{equation}
Finally, from $\left( \ref{DHP41} \right) $, $\left( \ref{DHP46} \right) $ and
$\left( \ref{DHP47} \right) $, by passing to the limit as $R\rightarrow
+\infty$ in $\left( \ref{DHP42} \right) $, by (\ref{pot1}) we obtain
\begin{equation}
E=\pi d^{2},
\end{equation}
which is $\left( \ref{DHP39} \right) $. \newline
Hence Theorem \ref{teo2}  is completely proved.


\begin{thebibliography}{99}
\bibitem {A} N. D. Alikakos, \emph{Some basic facts on the system $\Delta
u-W_{u}\left( u\right) =0 $}, Proc. Amer. Math. Soc., 139 (1) (2011), 153-162.

\bibitem {Am}L. Ambrosio and X. Cabr\'{e}, \emph{Entire solutions of semilinear elliptic equations in $R^{3}$ and a conjecture of De Giorgi}, 
J. Amer. Math. Soc., 13, (4) (2000), 725-739.



\bibitem {AS2}{N. André and I. Shafrir}, \emph{Asymptotic behaviour of minimizers for the Ginzburg-Landau functional with weight, Parts I and II},
Arch. Rat. Mech. and Anal., {142} (1) (1998), 45-73 and 75-98.


\bibitem {BH1}{A. Beaulieu and R. Hadiji}, \emph{Asymptotic for minimizers of a class of Ginzburg-Landau equation with weight}, C.R. Acad. Sci. Paris,
S\'er. I Math., {320} (2) (1995), 181-186.

\bibitem {BH2}{A. Beaulieu and R. Hadiji}, \emph{A Ginzburg-Landau problem having minima on the boundary}, 
Pro. Roy. Edinburgh Soc. A, {128} (1998), 123-148.

\bibitem {BH3}{A. Beaulieu and R. Hadiji}, \emph{Asymptotic behaviour of minimizers of a Ginzburg-Landau equation with weight near their zeroes},
Asymptotic Analysis, {22} (2000), 303-347.

\bibitem {BBH1}{F. Bethuel, H. Brezis and F. Hélein}, \emph{Asymptotic for the minimization of a Ginzburg-Landau functional}, 
Calculus of Variations and PDE, {1} (1993), 123-148.

\bibitem {BBH2}{F. Bethuel, H. Brezis and F. Hélein},
\emph{Ginzburg-Landau vortices, Birkh$\ddot{a}$user} (1994).

\bibitem {B}{H. Brezis},\emph{Comments on two Notes by L. Ma and X. Xu},
C. R. Acad. Sci. Paris, Ser. I 349 (2011), 269-271.



\bibitem {BMR}{H. Brezis, F. Merle and T. Rivière}, \emph{Quantization effects for $-\Delta u=u\left(  1-\left\vert u\right\vert ^{2}\right)  $ in $\R^{2}$}, 
Arch. Rat. Mech. Anal., {126} (1994), 35-58.

\bibitem {DG}{P. G. De Gennes}, \emph{Superconductivity of metals and alloys},
Benjamin, New York and Amsterdam, (1996).


\bibitem {DuG}{Q. Du and M. Gunzburger}, \emph{A model for supraconducting thin films having variable thickness}, Physica. D., {69} (1994), 215-231.

\bibitem {F1}{A. Farina}, \emph{ Two results on entire solutions of Ginzburg-Landau system in higher dimensions}, J. Funct. Anal., {214} (2) (2004), 386-395.

\bibitem {F2}{A. Farina}, \emph{ Finite-energy solutions, quantization effects and Liouville-type results for a variant of the Ginzburg-Landau systems in $R^{K}$}, 
Differential and Integral Equations, {11} (6) (1998), 875-893.


\bibitem{HH}
R.-M. Herv\'{e} and M. Herv\'{e},
\emph{\'{E}tude qualitative des solutions r\'{e}elles d'une \'{e}quation différentielle li\'{e}e \`a l'\'{e}quation de Ginzburg-Landau},
Ann. Inst. H. P. Anal. non Lin. (1994), 427-440.



\bibitem {HP2020} {R. Hadiji and  C. Perugia}, Minimization of a Ginzburg-Landau type energy with weight and with potential having a zero of infinite order, 
\emph{Mathematics}, 8 (997), (2020). 

\bibitem {HP} {R. Hadiji and C. Perugia}, \emph{Minimization of a quasi-linear Ginzburg-Landau type energy}, 
Nonlinear Analysis, {71} (2009), 860-875.

\bibitem {HS1}{R. Hadiji and I. Shafrir}, \emph{Minimization of a Ginzburg-Landau type energy with potential having a zero of infinite order}
Differential Integral Equations, {19} (10) (2006), 1157-1176.




\bibitem {LL}{ L. Lassoued and C. Lefter}, \emph{On a variant of the Ginzburg-Landau energy}, 
NoDEA Nonlinear Differential Equations Appl., 5 (1998), 39-51.

\bibitem {R} {J. Rubinstein}, \emph{On the equilibrium position of Ginzburg-Landau vortices}, 
Z. Angew Math. Phys., {46} (5) (1995), 739-751.

\bibitem {ES}{E. Sandier}, \emph{Lower bounds for the energy of unit vector fields and applications}, 
J. Functional Analysis, {152} (2) (1998), 379-403.

\bibitem {SaSe}{E. Sandier, S. Serfaty}, \emph{Vortices in the magnetic Ginzburg-Landau model}, 
Birkh$\ddot{a}$user, 
Basel 1997. 

\end{thebibliography}
\end{document}